\documentclass[11pt,leqno]{amsart}

\usepackage{amsfonts}                              %packages
\usepackage{amsmath}
\usepackage{amsthm}
\usepackage{amssymb}
\usepackage{diagrams}
\usepackage{bm}
\usepackage{graphics,color}
\usepackage{latexsym}
\usepackage{hyperref}

\theoremstyle{plain}                        %theorem environments
\newtheorem{Thm}{Theorem}[section]

\newtheorem{Cor}[Thm]{Corollary}
\theoremstyle{definition}

\newtheorem{Rem}[Thm]{Remark}

\numberwithin{equation}{section}
\DeclareMathOperator{\id}{id}                   %%operators
\DeclareMathOperator{\pr}{pr}

\newcommand{\col}{\colon}

\newcommand{\B}{\mathbb{B}}
\newcommand{\E}{\mathbb{E}}

\newcommand{\G}{\mathbb{G}}
\newcommand{\N}{\mathbb{N}}
\newcommand{\R}{\mathbb{R}}
\newcommand{\cA}{\mathcal{A}}
\newcommand{\cC}{\mathcal{C}}
\newcommand{\cG}{\mathcal{G}}
                                                        \newcommand{\cL}{\mathcal{L}}                                                         \newcommand{\cP}{\mathcal{P}}

\newcommand{\cU}{\mathcal{U}}

\newcommand{\ra}{\rightarrow}             %arrows
\newcommand{\lra}{\longrightarrow}

\newcommand{\xra}{\xrightarrow}

\newcommand{\fa}{\phi_{\gra}}
\newcommand{\fb}{\phi_{\grb}}
\newcommand{\gra}{\alpha}

\newcommand{\grb}{\beta}

\newcommand{\Grc}{\Gamma}
\newcommand{\grd}{\delta}
\newcommand{\grl}{\lambda}
\newcommand{\grL}{\Lambda}

\newcommand{\grs}{\sigma}
\newcommand{\Om}{\Omega^1_B}
\newcommand{\OmG}{\Omega^1_B(\G)}
\newcommand{\OmL}{\Omega^1(\cL)}
\newcommand{\sst}{\subseteq}

\newcommand{\Ad}{\mathrm{Ad}}
\newcommand{\Aut}{\mathrm{Aut}}
\newcommand{\ga}{g_{\gra}}
\newcommand{\gb}{g_{\grb}}
\newcommand{\ha}{h_{\gra}}
\newcommand{\hb}{h_{\grb}}
\newcommand{\ka}{k_{\gra}}
\newcommand{\xa}{\chi_{\gra}}
\newcommand{\gab}{g_{\gra\grb}}
\newcommand{\hab}{h_{\gra\grb}}
\newcommand{\oa}{\omega_{\gra}}
\newcommand{\ob}{\omega_{\grb}}
\newcommand{\ta}{\theta_{\gra}}
\newcommand{\tb}{\theta_{\grb}}
\newcommand{\ophi}{\overline{\phi}}

\newcommand{\sa}{s_{\gra}}
\newcommand{\sbe}{s_{\grb}}

\newcommand{\tea}{t_{\gra}}
\newcommand{\teb}{t_{\grb}}
\newcommand{\Ua}{U_{\gra}}
\newcommand{\Uab}{U_{\gra\grb}}

\newcommand{\ffh}{(f,\varphi,h)}             %principal bundles
\newcommand{\ffi}{(f,\varphi,\id_B)}

\newcommand{\pfb}{(P,G,B,\pi)}     %principal bundles
\newcommand{\qfb}{(Q,H,B',\pi')}

\newcommand{\wt}{\widetilde}

\newcommand{\vpl}{\varprojlim}
\begin{document}

\title{Local connection forms revisited}

\author{Efstathios Vassiliou}

\address{Department of Mathematics\\ University of Athens\\
Panepistimiopolis, Athens 15784, Greece}
\email{evassil@math.uoa.gr}


\begin{abstract}\noindent
Local connection forms provide a very useful tool for handling connections on principal bundles, because they ignore any complexities of the total space and, essentially, involve only two fundamental features of the structure group, namely the adjoint representation and the left (logarithmic) differential. The main results of this note characterize connections related together by bundle morphisms, while applications (taken from various sources) refer to connections on (Banach) associated bundles, in particular vector bundles, and connections on inverse limit bundles (in the Fr\'echet framework). The role of local connection  forms is further illustrated by their sheaf-theoretic globalization, resulting in a sheaf operator-like approach to principal connections. The latter point of view is naturally leading to a theory of connections on abstract principal sheaves.
\end{abstract}

%\today

\maketitle \thispagestyle{empty}

\section*{Introduction}                      %Introduction

\noindent In this note we are concerned with local connection forms  and their role in certain aspects of the theory of connections on principal bundles.

In the first place, we are dealing with connections mutually related by principal bundle morphisms. More precisely, let $\ffh$ be a morphism between the principal bundles $P\equiv\pfb$ and $Q\equiv\qfb$. Two connections $\omega$ and  $\theta$ on $P$ and $Q$, respectively, are said to be $\ffh$-related (or conjugate) if, roughly speaking, the morphism preserves the horizontal subspaces. Equivalent conditions are expressed also in terms of splittings of exact sequences and connection forms. Here our attention is focused on conditions involving local connection forms. This point of view is quite advantageous within the framework of infinite-dimensional bundles (e.g. Banach bundles, or Fr\'echet bundles derived as projective limits of Banach bundles), since local forms circumvent the (occasionally complicated) structure of the total space of the bundle, and involve only the base space and the structure group of the bundle. The latter group is essentially manifested by two of its fundamental features, namely the adjoint representation and the left  (Maurer-Cartan or logarithmic) differential.

The first three sections of the note are of preparatory character and provide an expository summary of connections on Banach principal bundles. The main results are given in Section~\ref{S4} and characterize various cases of related connections in terms of local connection forms. For the sake of completeness, Section~\ref{S5} is devoted to applications referring to connections on associated principal and vector bundles (\S\S\,\ref{Ss51}--\ref{Ss52}), while \S\,\ref{Ss53} is dealing with connections on Fr\'echet principal bundles, the latter obtained as projective limits of Banach principal bundles in the sense of \cite{Dod-Gal-Vas}, \cite{Gal2}. Finally, the significance of local connection forms is further illustrated by their sheaf-theoretic globalization (see \S\,\ref{Ss54}).  By virtue of this approach, ordinary principal connections correspond to sheaf-theoretic operators, in analogy to the sheaf-theoretic description of connections on vector bundles (cf., for instance, \cite{Wells}). Accordingly, the aforementioned criteria of relatedness acquire analogous sheaf-theoretic expressions. This point of view naturally leads to an abstract theory of connections on Grothendieck's principal sheaves (\cite{Grothendieck}), fully treated in \cite{Vas5}.

\section{Preliminaries} \label{S1}                 %section 1

Unless otherwise stated, all the manifolds and bundles mentioned throughout are modelled on Banach spaces, and differentiability is assumed to be of class $\cC^{\infty}$ (smoothness). For the basic definitions, notations, and terminology we mainly refer to Bourbaki \cite{Bourbaki1} and Lang \cite{Lang}.

A principal bundle will be denoted by $\pfb$, where the structure group $G$ acts on the total space $P$ from the right. Clearly $\pi\col P \ra B$ is the projection. If there is no danger of confusion, $\pfb$ will be simply denoted by $P$.

To a principal bundle $P$, as before, we associate the exact sequence of vector bundles (over $P$):
\begin{equation}   \label{Eq11}                 %(1.1)
   0\lra P\times \mathfrak{g} \xra{\;\;\nu\;\;} TP \xra{\;\;T\pi !\;\;}\pi^*(TB)\lra 0
\end{equation}

\noindent Here, $P\times \mathfrak{g}$ is the trivial vector bundle of fiber type $\mathfrak{g}$, where $\mathfrak{g}\equiv T_eG$ denotes the Lie algebra of $G$. $TP$ is the tangent bundle of $P$ and $\pi^*(TB)$ the pull-back of $TB$ by $\pi$.

The vector bundle morphism (\emph{vb-morphism}, for short) $\nu$ is defined by
\begin{equation}   \label{nu}                       %(1.2)
    \nu(p,X):=X^*_p; \qquad (p,X)\in P\times \mathfrak{g},
\end{equation}
where $X^*$ is the \emph{fundamental (Killing) vector field} corresponding to $X\equiv X_e\in \mathfrak{g}$ by the (right) action $\grd \col P\times G\ra G$. We recall that
\begin{equation}  \label{Eq13}                     %(1.3)
          X^*_p:= T_e\wt{p}(X_e):= T_e\grd_p(X_e).
\end{equation}
Note that, for every $p\in P$, the partial map $\wt p := \grd_p\col G\ra \pi^{-1}(x)$
is a diffeomorphism, with $x=\pi(p)$. The morphism $\nu$ is an immersion and its image $\nu(P\times \mathfrak{g})=VP$ is the \emph{vertical subbundle} of the tangent bundle $(TP,\tau_P,P)$. The fiber $V_pP$ is the \emph{vertical} subspace of $T_pP$, given by
\begin{equation} \label{Eq14}                      %(1.4)
  V_pP \equiv (VP)_p = \ker(T_p\pi !) = \ker(T_p\pi) =
                                            T_p(\pi^{-1}(x)).
\end{equation}

The vb-morphism $T\pi!=(\tau_P,T\pi)$ is induced by the universal property of the pull-back, as pictured in the next diagram.

\begin{diagram}                             \label{Dgr1}
   TP &                   &            &              &    \\
   &\rdTo~{\rotatebox{45}{$T\pi$\! !}}\rdTo(4,2)^{\raisebox{-1ex}
                                {\rotatebox{20}{$T\pi$}}}
   \rdTo(2,4)_{\rotatebox{60}{$\tau_{\mbox{\tiny$P$}}$}}& &  &  \\
    &         &\pi^*(TB)&\rTo_{\qquad \pi^* =\pr_2\quad\qquad} &  TB  \\
    &       &         \dTo_{\tau_{\mbox{\tiny$B$}}^*=\pr_1} &
         &\dTo_{\tau_{\mbox{\tiny$B$}}}  \\
    &       &        P  & \rTo^\pi  &  B
\end{diagram}

\medskip\centerline{{\sc Diagram~1}}

\bigskip The structure group $G$ acts naturally on the bundles of \eqref{Eq11}. More precisely, for every $g\in G$,
\begin{equation}  \label{Eq15}                  %(1.5)
\begin{aligned}
    (p,X)\cdot g &:= \left(pg, \Ad(g^{-1})(X)\right); &
                    \quad (p,X)& \in P\times \mathfrak{g},\\
          u\cdot g &:= TR_g(u); & \quad u &\in TP,\\                            (p,v)\cdot g &:= (pg, v); & \quad (p,v)&\in \pi^*(TB).
\end{aligned}
\end{equation}

\noindent  As usual,$R_g\col P\ra P$ denotes the right translation of $P$ by $g\in G$, i.e. $R_g(p) := \grd(p,g)\equiv p\cdot g\equiv pg$. Therefore, taking into account the definition of $T\pi !$,  it follows that the vb-morphisms $\nu$ and $T\pi !$  are \emph{ $G$-equivariant}; in other words,
\begin{equation}  \label{Eq16}                   %(1.6)
  \begin{gathered}
   \nu\big((p,X_e)\cdot g\big) = T_pR_g\big((\nu(p,X_e)\big)=
                                          \nu(p,X_e)\cdot g, \\
   T\pi !(u\cdot g) = \big(pg, T\pi(u)\big)=
                 \big(p, T\pi(u)\big)\cdot g = T\pi!(u)\cdot g,
  \end{gathered}
\end{equation}
for every $(p,X_e)\in P\times \mathfrak{g}$, $u\in T_pP$ and every $g\in G$.

The previous considerations characterize \eqref{Eq11} as an (exact) sequence of \emph{$G$-vector bundles}.

\section{Connections on principal bundles}\label{S2} %section 2

Let $P\equiv\pfb$ be a principal bundle. A \emph{connection} on $P$ is a $G$-splitting of the exact sequence \eqref{Eq11}; that is, an exact sequence of $G$-vector bundles
\begin{equation} \label{Eq21}                     %(2.1)
0\lra \pi^*(TB) \xra{\,\;\;C \;\;} TE \xra{\,\;\;V\;\;}
                                     P\times\mathfrak{g} \lra  0
\end{equation}
such that $C$ and $V$ are $G$-equivariant morphisms satisfying
\begin{equation}        \label{Eq22}              %(2.2)
    T\pi !\circ C = \id_{\pi^*(TB)},
                                 \quad V\circ \nu = \id_{VE}
\end{equation}
(see also \cite{Penot}). Of course, it suffices to know either $C$ or $V$ to determine the splitting sequence \eqref{Eq21}. The next diagram illustrates the previous definition.

\begin{diagram}                     \label{Dgr2}
   &      &    &       &  0   &      &      &       & \\
   &      &    &       &  \dTo_{}    &      &       &   & \\
   &      &    &       &  \pi^*(TB)  &      &       &   & \\
   &      &    &       & \dTo_{C}    &
    \rdTo^{\rotatebox{40}{\raisebox{-15ex}{$\!\id_{\pi^*(TB)}$}}}
                                             &       &    & \\
 0 & \rTo^{} & P\times \mathfrak{g} & \rTo^{\;\;\nu\;\;} & TP  &
                                       \rTo^{\;\;T\pi !\quad} &
                                \pi^*(TB) & \rTo^{}   & 0   \\
   &      &    & \rdTo_{\rotatebox{45}{$\id_{P\times \mathfrak{g}}$}}   &
                       \dTo_{V}   &       &    &     & \\
   &      &    &       & P\times \mathfrak{g}  &  &       &   & \\                                     &      &    &       &  \dTo^{}     &      &       &   & \\
   &      &    &       &  0           &      &       &   &
\end{diagram}

\medskip\centerline{{\sc Diagram~2}}

\bigskip The splitting determines the  \emph{horizontal subbundle} $HE := \mathrm{Im}(C)$ of $TP$  completing the vertical subbundle, i.e.
\begin{equation} \label{Eq23}                      %(2.3)
                   TP = VP \oplus HP.
\end{equation}

\medskip\noindent It follows that      %(2.4)-(2.5)
\begin{gather}
   HP =C\left(\pi^*(TB)\right) = \ker V, \label{Eq24}\\
   T_pR_g(H_pP) = H_{pg}P,   \label{Eq25}
\end{gather}
for every $p\in P$ and $g\in G$.  Moreover, from the definition of $T\pi!$ and the first of \eqref{Eq22}, we see that the horizontal component $u^h$ of every $u\in TP$ has the form
\begin{equation}  \label{Eq26}                      %(2.6)
    u^h = C\big(\tau_{\mbox{\tiny $P$}}(u),T\pi(u)\big),
\end{equation}
after the decomposition \eqref{Eq23}.

\smallskip Equivalently, connections on principal bundles are defined by global connection forms. First we recall that, for a principal bundle as before, $L(TP,P\times \mathfrak{g})$ denotes the linear map bundle, whose fiber over a point $p\in P$ is the (Banach) space of continuous linear maps $\cL(T_pP,\mathfrak{g})$. The smooth sections of the previous bundle are the \emph{$\mathfrak{g}$-valued 1-forms} on $P$, whose set is usually denoted by $\grL^1(P,\mathfrak{g})$.

 A \emph{connection form} of $P$ is an element $\omega \in \grL^1(P,\mathfrak{g})$ satisfying the following conditions:
\begin{gather}
\omega(X^{*}) = X; \qquad X\in \mathfrak{g},
                                  \tag{$\omega.~1$}\label{w1} \\
 R^{*}_{g}\omega = \Ad(g^{-1}).\omega; \qquad g\in G.
                                  \tag{$\omega.~2$} \label{w2}
\end{gather}
Evaluated at any $p\in P$, the preceding equalities  yield respectively:
\begin{gather*}
  \omega_p(X^*_p)= X\equiv X_e,\\
  \omega_{pg}(T_pR_g(u)) = \Ad(g^{-1}).\omega_p(u),
\end{gather*}
for every $X\in \mathfrak{g}$, $g\in G$ and $u\in T_pP$. Line dots, as in the last equality, quite often replace consecutive parentheses and should not be confused with center dots customarily indicating the result of a group multiplication or action.

A connection form $\omega$ is related with the morphism $V$ of the splitting sequence \eqref{Eq21} by
\[
      \omega_p(u) = (\pr_2\circ V)(u); \qquad p\in P,\, u\in T_pP.
\]
Therefore, \eqref{Eq24} implies that
\begin{equation} \label{Eq27}                        %(2.7)
    H_pP = \ker \omega_p, \qquad p\in P.
\end{equation}

If $u\in T_pP$, then we write $u=u^v+u^h$, with $u^v$ denoting the vertical components of $u$. If we set $\omega_p(u)=A$, then \eqref{Eq27} implies that $\omega_p(u^v)=A$. Since $\omega_p(A^*_p) = A$, it follows that $\omega_p(A^*_p-u^v)=0$, or $V_pP\ni A^*_p-u^v\in H_pP$, thus $A^*_p=u^v$. In a descriptive way (see also \cite{Kob-Nom}),  $\omega_p(u)$ identifies with the element of $\mathfrak{g}$ whose corresponding fundamental vector field  coincides, at $p$, with the vertical component of $u$.

Finally, $\omega$ is related with the splitting morphism $C$ by
\begin{equation} \label{Eq28}                        %(2.8)
    C(p,v) = u-\nu(\omega_p(u)); \qquad (p,v)\in \pi^*(TB),
\end{equation}
for any $u\in T_pP$ with $T_p\pi(u)=v$.

\section{Local connection forms} \label{S3}     %section 3

As mentioned in the introduction, local connection forms are very convenient because they involve only the base space and the structure group of the bundle. To define them we need a few notations regarding the local structure of a principal bundle.

Let $\cC= \{(U_{\gra},\Phi_{\gra})\,|\, \gra\in I\}$ be a trivializing cover of $\pfb$, where  $\Phi_{\gra}\col \pi^{-1}(U_{\gra})\ra U\times G$ is a $G$-equivariant diffeomorphism.  The \emph{natural sections} (with respect to $\cC$) are defined to be the smooth maps $\sa(x):=\Phi^{-1}_{\gra}(x,e)$, for every $x\in U_{\gra}$. The corresponding \emph{transition functions}, denoted by
$\gab \col\Uab \ra G$ ($\Uab:=U_{\gra}\cap U_{\grb} $), are given by
$\sbe=\sa\cdot \gab$, for all $\gra,\grb\in I$.

Assume now that $P$ admits a connection with connection form  $\omega\in \grL^1(P,\mathfrak{g})$. Then the 1-forms
\begin{equation}  \label{locconforms}         %(3.1)
  \oa := \sa^*\omega \in \grL^1(U_{\gra},\mathfrak{g});
                                           \qquad \gra \in I,
\end{equation}
are the \emph{local connection forms} of the given connection $(\equiv \omega)$, with respect to $\cC$. Therefore,
\[
    \omega_{\gra,x}(v) = \omega_{\sa(x)}(T_x\sa(v));
                  \qquad x\in U_{\gra},\,  v\in T_xU_{\gra}\equiv T_xB.
\]
The forms $\{\omega_{\gra}\}$ satisfy the compatibility condition
\begin{equation}  \label{loccompcond}        %(3.2)
 \ob = \Ad(\gab^{-1}).\oa + \gab^{-1}d\gab; \qquad \gra,\grb \in I.
\end{equation}
over $\Uab$. The second summand on the right-hand side of \eqref{loccompcond} is the \emph{left Maurer-Cartan differential} of $\gab\in \cC^{\infty}(U_{\gra\grb},G)$.

For later use, it is useful to recall that the left Maurer-Cartan differential of a smooth map $f\col U\ra G$ ($U\sst B$ open) is defined by
\begin{equation} \label{Eq33}                   %(3.3)
    \left(f^{-1}df\right)_x(v):= \left(T_{f(x)}\grl_{f(x)^{-1}}
   \circ T_xf\right)(v); \qquad  x\in U,\, v\in T_xB,
\end{equation}
where $\grl_g\col G\ra G$ denotes the left translation of $G$ by $g\in G$. Concerning our terminology, this is justified by equality
\begin{equation}    \label{Eq34}                %(3.4)
    f^{-1}df = f^*\gra^{\ell}
\end{equation}
where $\gra^{\ell}$ is the left Maurer-Cartan form of $G$.
Other terms in use for $f^{-1}df$ are \emph{left differential} (\cite[Ch.~III, \S\,3.17]{Bourbaki2}), \emph{logarithmic differential} (\cite[Ch.~VIII, \S\,38.1]{Mich}), or \emph{multiplicative differential} (\cite[Ch.~I, \S\,3]{Krein}) of $f$.

Conversely, let $\cC= \{(U_{\gra},\Phi_{\gra})\,|\, \gra\in I\}$ be a trivializing cover of a principal bundle $P$. A family of 1-forms $\{\oa \in  \grL^1(U_{\gra},\mathfrak{g})\,|\, \gra \in I\}$, satisfying the compatibility condition \eqref{loccompcond}, determines a unique connection form $\omega\in \grL^1(P,\mathfrak{g})$, whose local connection forms are precisely the given $\{\oa\}$. Briefly, the construction of $\omega$ goes as follows: For each $\gra\in I$, we define the map $\ga \col \pi^{-1}(U_{\gra}) \ra G$, \label{ga} by
setting $\ga(p) =(\pr_2\circ \Phi_{\gra})(p)$; hence, $\ga$ is a smooth map such that
\[
  p=\sa(\pi(p))\cdot \ga(p), \qquad  p\in \pi^{-1}(U_{\gra})
\]
(the preceding equality can be also used as a definition of $\ga$) . Then, for every $p\in P$ with $\pi(p)\in U_{\gra}$, and every $u\in T_pP$, we set
\begin{equation}   \label{Eq35}                    %(3.5)
 \omega_p(u):=\Ad\left(\ga(p)^{-1}\right)\!.(\pi^*\omega_{\gra})_p(u)
                       + \left(\ga^{-1}d\ga\right)_p(u).
\end{equation}
Condition \eqref{loccompcond} ensures that $\omega$ is a well defined element of $\grL^1(P,\mathfrak{g})$ which determines a connection form.

For the analog of \eqref{Eq35} in the case of a Lie group $G$ acting on the left of $P$ see \cite[p.~129]{Sulanke}. We also refer to  \cite[pp.~32--33]{Bleecker} and \cite[pp.~227--228]{Pham} for other ways to define $\omega$ from $\{\oa\}$.

The proof of \eqref{loccompcond} and the fact that \eqref{Eq35} is a well defined form satisfying \eqref{w1} -- \eqref{w2}, are based on certain computations which will be of frequent use below. More precisely, assume that $\grs$ and $s$ are two sections of $P$ over the same open subset $U$ of $B$. Then there is a unique smooth map $g\col U\ra G$ such that $\grs = sg = \grd \circ (s, g)$. Hence, for every $x\in U$ and $v\in T_xB$,
\begin{align*}
    T_x\grs(v) &= T_{s(x)}\grd_{g(x)}(T_xs(v)) +
                                T_{g(x)}\grd_{s(x)}(T_xg(v))\\
    &= T_{s(x)}R_{g(x)}(T_xs(v)) +
              \big(T_e\grd_{s(x)\cdot g(x)} \circ
                          T_{g(x)}\grl_{g(x)^{-1}}\big)(T_xg(v))\\
    &= T_{s(x)}R_{g(x)}(T_xs(v)) +
            T_e\grd_{\grs(x)}\big((g^{-1}dg)_x(v)\big).
\end{align*}
Setting $(g^{-1}dg)_x(v)=A\in \mathfrak{g}$, \eqref{Eq13} implies that
\[
T_e\grd_{\grs(x)}\big((g^{-1}dg)_x(v)\big)=A^*_{\grs(x)},
\]
consequently
\[
  T_x\grs(v) = T_{s(x)}R_{g(x)}(T_xs(v)) + A^*_{\grs(x)}.
\]
Applying $\omega$ to the latter equality, we obtain:
\begin{align*}
    \omega_{\grs(x)}(T_x\grs(v)) &=
           (R_{g(x)}^*\omega)_{s(x)}(T_x\grs(v))
                             + \omega_{\grs(x)}(A^*_{\grs(x)})\\
      &=  \Ad(g(x)^{-1}).\omega_{s(x)}(T_xs(v)) + A\\
      &= \Ad(g(x)^{-1}).\omega_{s(x)}(T_xs(v)) + (g^{-1}dg)_x(v).
\end{align*}
Equivalently, for every $x\in U$ and $v\in T_xB$,
\begin{equation} \label{Eq36}                     %(3.6)
   (\grs^*\omega)_x(v) = \Ad(g(x)^{-1}).(s^*\omega)_x(v) +
                                                 (g^{-1}dg)_x(v).
\end{equation}
In brief, we write:
\begin{equation} \label{Eq37}                     %(3.7)
      \grs^*\omega = \Ad(g^{-1}).s^*\omega + g^{-1}dg,
\end{equation}
for every $\grs,s\in \Gamma(U,P)$ with $\grs =sg$.

\section{Related connections}  \label{S4}   %section 4

We come now to the main theme and results of the present note.

\smallskip Let $P\equiv(P,G,B,\pi)$ and $Q\equiv(Q,H,B',\pi')$ be principal bundles endowed with the connections $\omega$ and $\theta$, respectively. We denote by  $C$, $V$ the splittings of the exact sequence \eqref{Eq21} corresponding to $P$ and $\omega$. Analogously, $C'$ and $V'$ are the splittings corresponding to $Q$ and $\theta$.

If $\ffh$ is a principal bundle morphism (\emph{pb-morphism}, for short) of $P$ into $Q$, we denote by $\overline\varphi \col \mathfrak{g}\ra\mathfrak{h}$ the Lie algebra morphism induced by $\varphi$. Then $\omega$ and $\theta$ are called \emph{$\ffh$-related} (or, \emph{$(f,\varphi,h)$-conjugate}) if one of the following equivalent conditions holds (see also \cite{Vas1}):
\begin{equation} \label{Eq41}                       %(4.1)
 \begin{gathered}
     f^*\theta = \overline\varphi \omega,\\
     Tf(u^v) = \left(Tf(u)\right)^{v'},\\
     Tf(u^h) = \left(Tf(u)\right)^{h'},\\
     V'\circ Tf = C\circ (f\times \overline\varphi),\\
     C'\circ (f\times Th) =  Tf \circ C,
\end{gathered}
\end{equation}
for every $u\in T_pP$. More explicitly, the first condition  means that
\[
  (f^*\theta)_p(u)=\theta_{f(p)}(T_pf(u)) =
  \overline\varphi(\omega_p(u)); \qquad p\in P,\, u\in T_pP.
\]
We recall that the superscripts $v$ and $h$ indicate, respectively, the vertical and horizontal components of any $u\in TP$, with respect to $\omega$. Analogous considerations hold for their dashed counterparts referring to $\theta$.

Assume now that $\pfb$ and $(Q,H,B,\pi')$ are principal bundles over the \emph{same base}, admitting the respective connections $\omega$ and $\theta$. Taking trivializations over the same open cover $\{U_{\gra}\,|\,\gra\in I\}$ of $B$, we obtain the natural local sections $\{\sa\}_{\gra\in I}$ and $\{\tea\}_{\gra\in I}$ of $P$ and $Q$, respectively, as  well as the local connection forms $\{\oa\}_{\gra\in I}$ and $\{\ta\}_{\gra\in I}$ corresponding to $\omega$ and $\theta$.

With these notations, we  prove the following criterion of relatedness of connections in terms of local connection forms:

\begin{Thm}\label{T41}                  %theorem 4.1
Let $\ffi$ be a pb-morphism of $P\equiv \pfb$ into $Q\equiv (Q,H,B,\pi')$.
Two connections $\omega$ and $\theta$ on $P$ and $Q$, respectively, are $(f,\varphi,\id_B)$-related if and only if
\begin{equation}   \label{Eq42}               %(4.2)
 \overline\varphi\omega_{\gra} = \Ad(h^{-1}_{\gra}).\theta_{\gra}
            + h^{-1}_{\gra}dh_{\gra}; \qquad \gra,\grb\in I,
\end{equation}
where $\{\ha\col U_{\gra} \ra H\,|\,\gra\in I\}$ are smooth maps satisfying equality
\[
   f\circ \sa = \tea\cdot \ha, \qquad  \gra\in I.
\]
\end{Thm}

\begin{proof} Assume first that $\omega$ and $\theta$ are $(f,\varphi,\id_B)$-related, i.e. $f^*\theta = \overline\varphi \omega$.  Then, for every $x\in U_{\gra}$ and $v\in T_xB$,
the analog of \eqref{Eq36} for $f\circ \sa$ and $\tea$ yields
\begin{equation}  \label{Eq43}                %(4.3)
  \big(\sa^*(f^*\theta)\big)_x(v) =
    \Ad\left(\ha(x)^{-1}\right). \theta_{\gra,x}(v)+
                           \left(\ha^{-1}dh_{\gra}\right)_x(v).
\end{equation}
On the other hand,
\begin{equation}  \label{Eq44}           %(4.4)
 \big(\sa^*(\overline\varphi\omega)\big)_x(v) =
       \overline\varphi\big((\sa^*\omega)_x(v)\big) =
                \overline\varphi\big(\omega_{\gra,x}(v)\big).
\end{equation}
Hence, in virtue of the assumption, equalities \eqref{Eq43} and \eqref{Eq44} lead to \eqref{Eq42}.

Conversely, assume that condition \eqref{Eq42} holds. Then, on each $\pi^{-1}(U_{\gra})$, we determine a smooth map $\ga\col \pi^{-1}(U_{\gra}) \ra G$  such that $p=\sa(\pi(p))\cdot \ga(p)$, for every $p\in \pi^{-1}(U_{\gra})$. By routine computations we see that
\begin{gather}                  %(4.5}-(4.6}
 \overline\varphi\circ \Ad\big(\ga(p)^{-1}\big)=\Ad\big((\varphi\circ
   \ga)(p)^{-1}\big)\circ \overline\varphi, \label{Eq45}\\
 \overline\varphi(\ga^{-1}d\ga) = (\varphi\circ \ga)^{-1}d(\varphi
                          \circ \ga).      \label{Eq46}
\end{gather}
Thus, for every $p\in \pi^{-1}(U_{\gra})$ and every $u\in T_pP$, equalities \eqref{Eq45} and \eqref{Eq46} applied to \eqref{Eq35} yield
\begin{align*}
 \overline\varphi(\omega_p(u)) &= \overline\varphi\circ \Ad(\ga(p)^{-1})
    .(\pi^*\oa)_p(u)+\overline\varphi.(\ga^{-1}d\ga)_p(u)\\
  &= \Ad\big(\varphi(\ga(p))^{-1}\big)\circ \overline\varphi
    .(\pi^*\oa)_p(u) + \big((\varphi\circ\ga)^{-1}d(\varphi\circ\ga)\big)_p(u)\\
  &= \Ad\big(\varphi(\ga(p))^{-1}\big)\circ \overline\varphi
    .\omega_{\gra,x}(T_p\pi(u)) + \big((\varphi\circ\ga)^{-1}d(\varphi\circ\ga)\big)_p(u),
\end{align*}
or, in virtue of \eqref{Eq42} and $x=\pi(p)$,
\begin{equation}      \label{Eq47}              %(4.7)
  \begin{aligned}
    \overline\varphi(\omega_p(u))
    &= \Ad\big(\varphi(\ga(p))^{-1}\big)[\Ad(\ha(x)^{-1}).
                                      \theta_{\gra,x}(T_p\pi(u))
    +(\ha^{-1}d\ha)_p(u)]\\
  &\quad + \big((\varphi\circ\ga)^{-1}d(\varphi\circ\ga)\big)_p(u).
  \end{aligned}
\end{equation}

To express $\theta$ by local connection  forms [analogously to \eqref{Eq35}], we define the smooth maps $\ka\col \pi'^{-1}(U_{\gra}) \ra H$ ($\gra\in I$) such that $q = \tea(\pi'(q))\cdot \ka(q)$, for every $q\in \pi'^{-1}(U_{\gra})$. Then, for $q=f(p)$, we obtain
\[
   f(p) = \tea(\pi'(f(p))\cdot \ka(f(p)) =
                                    \tea(x)\cdot \ka(f(p)).
\]
Since, by the definition of $\ga$,
\[
   f(p)=f(\sa(x)\cdot\ga(p)) = f(\sa(x))\cdot\varphi(\ga(p)))=
      \tea(x)\cdot \ha(x)\cdot\varphi(\ga(p)),
\]
it follows that $\ka(f(p))=\ha(x)\cdot\varphi(\ga(p))$, or
\begin{equation}   \label{Eq48}                  %(4.8)
  \ka\circ f|_{\pi^{-1}(U_{\gra})}=
                            (\ha\circ \pi)\cdot(\varphi\circ\ga).
\end{equation}
Therefore,
\begin{equation}  \label{Eq49}                   %(4.9)
 \begin{aligned}
 & (f^*\theta)_p(u) = \theta_{f(p)}(T_pf(u))\\
 &= \Ad\big(\ka(f(p))^{-1}\big).(\pi'^*\ta)_{f(p)}(T_pf(u))+
                                   (\ka^{-1}d\ka)_{f(p)}(T_pf(u))\\
 &= \Ad\big(\ka(f(p))^{-1}\big).(\pi'^*\ta)_{f(p)}(T_pf(u))+
                 \big((\ka\circ f)^{-1}d(\ka\circ f)\big)_p(u).
 \end{aligned}
\end{equation}                                                                                                        Because the Maurer-Cartan differential of the right-hand side of  \eqref{Eq48} is given by
\begin{gather*}
  \big((\ha\circ \pi)\cdot(\varphi\circ\ga)\big)^{-1}
   d\big((\ha\circ \pi)\cdot(\varphi\circ\ga)\big)^{-1} =\\
       = (\varphi\circ\ga)^{-1}d(\varphi\circ\ga) +
                   \Ad\big((\varphi\circ\ga)^{-1}\big).
                        (\ha\circ \pi)^{-1}d(\ha\circ \pi),
\end{gather*}
applying \eqref{Eq48} and the latter differential, we transform \eqref{Eq49} into
\begin{align*}
(f^*\theta)_p(u) &=
 \Ad\big(\varphi(\ga(p))^{-1}\big)\circ
     \Ad\big(\ha(x)^{-1}\big).\theta_{\gra,x}(T_p\pi(u))\\
     & \qquad +
       \big[\big((\ha\circ\pi)\cdot(\varphi\circ\ga)\big)^{-1}
        d\big((\ha\circ\pi)\cdot(\varphi\circ\ga)\big)\big]_p(u)\\
 &= \Ad\big(\varphi(\ga(p))^{-1}\big)\circ
       \Ad(\ha(x)^{-1}).\theta_{\gra,x}(T_p\pi(u))\\
 &\qquad +
      \big((\varphi\circ\ga)^{-1}d((\varphi\circ\ga))\big)_p(u)\\
 & \qquad + \Ad\big(\varphi(\ga(p)^{-1}\big).
          \big((\ha\circ \pi)^{-1}d (\ha\circ \pi)\big)_p(u)\\
 &= \Ad\big(\varphi(\ga(p))^{-1}\big)\big[\Ad(\ha(x)^{-1}.
      \theta_{\gra,x}(T_p\pi(u)) + (\ha^{-1}d\ha)_x(u)\big]\\
 & \qquad+\big((\varphi\circ\ga)^{-1}d((\varphi\circ\ga))\big)_p(u).
\end{align*}
Comparing the preceding with \eqref{Eq47}, we establish \eqref{Eq42} and conclude the proof.
\end{proof}

A variation of Theorem~\ref{T41} assures even the construction of the morphism between the principal bundles under consideration, provided that a Lie group morphism between their structure groups is given beforehand. Indeed, with the assumptions mentioned before the statement of Theorem~\ref{T41} still in force, we obtain:

\begin{Thm}\label{T42}                %theorem 4.2
Let $\omega$ and $\theta$ be connections on $P\equiv \pfb$ and $Q\equiv (Q,H,B,\pi'$, respectively, and let $\varphi\col G\ra H$ be a given morphism of (Banach-) Lie groups. Then $\omega$ and $\theta$ are related by a principal bundle morphism $\ffi$ if and only if \eqref{Eq42} holds along with condition
\begin{equation}   \label{Eq410}               %(4.10)
 \hab = \ha\cdot (\varphi\circ \gab)\cdot \hb^{-1},
\end{equation}
where $\{\gab\}$  and $\{\hab\}$ are the transition functions of $P$ and $Q$, respectively, and $\{\ha\}$ are the smooth maps of Theorem~\ref{T41}.
\end{Thm}

\begin{proof}
If there exists a pb-morphism $(f,\varphi,\id_B)$ such that $\omega$ and $\theta$ are $(f,\varphi,\id_B)$-related, then by the previous theorem we obtain \eqref{Eq42}.

Moreover, the definition of $\{\ha\}$ implies that $f(s_i(x))=t_i(x)\cdot h_i(x)$, for every $x\in U_i$ and $i=\gra,\grb$. Since $\sbe(x)=\sa(x)\cdot \gab(x)$ and $\teb(x)=\tea(x)\cdot \hab(x)$, for all $x\in \Uab$, it follows that
\[
  \teb(x)\cdot \hb(x) = f(\sbe(x)) = f(\sa(x)\cdot\gab(x)) =
                                     f(\sa(x))\cdot\varphi(\gab(x)),
\]
or, equivalently,
\[
  \tea(x)\cdot\hab(x)\cdot\hb(x) =
                      \tea(x)\cdot\ha(x)\cdot\varphi(\gab(x)),
\]
from which \eqref{Eq410} follows.

Conversely, assume that conditions \eqref{Eq410} and \eqref{Eq42} hold. The first suffices to define the morphism  $(f,\varphi,\id_B)$, thus \eqref{Eq42} will imply the relatedness of the connections by the inverse part of Theorem \ref{T41}. As a matter of fact, we define the map $f\col P\ra Q$ by setting
\begin{equation}   \label{Eq411}               %(4.11)
   f(p) := \tea(\pi(p))\cdot\ha(\pi(p))\cdot\varphi(\ga(p)),
\end{equation}
for every $p\in\pi^{-1}(U_{\gra})$, where $\ga\col \pi^{-1}(U_{\gra})\ra G$ is the smooth map satisfying equality $p=\sa(\pi(p))\cdot \ga(p)$, for every $p\in \pi^{-1}(U_{\gra})$ (see the proof of Theorem~\ref{T41}).

We claim that $f$ is well defined: For every $p\in \pi^{-1}(U_{\grb})$, we have the analogous expression
\begin{equation}   \label{Eq412}               %(4.12)
   f(p) = \teb(\pi(p)\cdot\hb(\pi(p))\cdot\varphi(\gb(p)),
\end{equation}
thus we need to show that \eqref{Eq411} and \eqref{Eq412} coincide on overlappings. Indeed, for every $p\in\pi^{-1}(\Uab)$,
\[
   \sbe(\pi(p))\cdot \gb(\pi(p)) = p = \sa(\pi(p))\cdot\ga(p),
\]
or
\[
  \sa(\pi(p))\cdot\gab(\pi(p)) \cdot \gb(\pi(p)) =
                                      \sa(\pi(p))\cdot\ga(p);
\]
hence,
\[
      \gb(p)=g_{\grb\gra}\cdot \ga(p),
\]
and, by applying $\varphi$,
\begin{equation}   \label{Eq413}                 %(4.13)
  \varphi(\gb(p))=\varphi(g_{\grb\gra})\cdot \varphi(\ga(p)).
\end{equation}
Therefore, in virtue of \eqref{Eq410}, \eqref{Eq413}, the analog of the right-hand side of \eqref{Eq411} for $\tea$ transforms into
\begin{gather*}
 \teb(\pi(p))\cdot\hb(\pi(p))\cdot\varphi(\gb(p))=\\
 \tea(\pi(p))\cdot\hab(\pi(p))\cdot\hb(\pi(p))\cdot
                                     \varphi(\gb(p))=\\
 \tea(\pi(p))\cdot\ha(\pi(p))\cdot\varphi(\gab(\pi(p)))
                                 \cdot\varphi(\gb(p))=\\
 \tea(\pi(p))\cdot\ha(\pi(p))\cdot\varphi(\ga(p)),
\end{gather*}
which proves the claim. The smoothness of $f$ is clear from the definition of $f|_{\pi^{-1}(U_{\gra})}$, for every $\gra\in I$.

Besides, for every $p\in \pi^{-1}(U_{\gra})$ and $\gra\in I$,
\[
  (\pi'\circ f)(p) = \pi'\big(\tea(\pi(p))\cdot\ha(\pi(p))
   \cdot\varphi(\ga(p))\big)= \pi'(\tea(\pi(p))= \pi(p),
\]
i.e. $\pi'\circ f=\pi$. Finally, we show that $f(pg)=f(p)\varphi(g)$, for every $p\in P$ and $g\in G$: If, for instance, $p\in \pi^{-1}(U_{\gra})$, then, analogously to \eqref{Eq411}, we may write
\begin{align*}
 f(pg) &= \tea(\pi(pg)) \cdot\ha(\pi(pg))
                                \cdot\varphi(\xa(pg))\\
 &= \tea(\pi(p)) \cdot\ha(\pi(p))\cdot\varphi(\xa(pg)),
\end{align*}
where $\xa\col \pi^{-1}(U_{\gra})\ra G$ is the smooth map defined by
\[
  pg = \sa(\pi(pg)) \cdot \xa(pg) =
                           \sa(\pi(p)) \cdot \xa(pg);
\]
But $p=\sa(\pi(p))\cdot\ga(p)$ implies that
\[
   \sa(\pi(p))\cdot\ga(p)\cdot g =
                             \sa(\pi(p)) \cdot \xa(pg),
\]
from which follows that $\xa(pg)=\ga(p) g$. The latter leads to
\[
  f(pg)=
     \tea(\pi(p)) \cdot\ha(\pi(p))\cdot\varphi(\xa(pg))=
                                                       f(p)\cdot\varphi(g),
\]
as desired. We thus conclude that $\ffi$ is a pb-morphism. Obviously, the very definition of $f$ implies that $f(\sa(x))=\tea(x)\cdot\ha(x)$, for every $x\in U_{\gra}$. The proof is now complete, as we explained in the beginning of the inverse part of the theorem.
\end{proof}

\begin{Rem}  \label{R43}                 %remark 4.3
The previous theorems can be easily extended  to principal bundles over diffeomorphic bases $B$ and $B'$, respectively. In this case we may consider morphisms of the form $(f,\varphi,f_o)$, where $f_o\col B\ra B'$ is a diffeomorphism. This would lead to undue complications in the computations (although the proofs follow the same patterns as before), without a particular gain in the final outcome.

However, for the sake of completeness, we briefly indicate the main steps of this approach: First of all we should rearrange the trivializations of the bundles so that we can consider the trivializing covers $\{(U_{\gra},\Phi_{\gra})\}_{\gra\in I}$ and  $\{(f_o(U_{\gra}),\Psi_{\gra})\}_{\gra\in I}$, of $P$ and $Q$, respectively. Then \eqref{Eq42}, \eqref{Eq410} and \eqref{Eq411}               take the respective forms
\begin{equation}
 \overline\varphi.\omega_{\gra,x}(v) = \Ad\big(h^{-1}_{\gra}(x)\big).
      \theta_{\gra,f_o(x)}(T_xf_o(v)) +
   \left(h^{-1}_{\gra}dh_{\gra}\right)_x(v),  \tag{4.2$'$}
\end{equation}
for every $x\in \Uab, v\in T_xB$, and every $\gra,\grb\in I$;
\begin{equation}
 \hab(f_o(x)) = \ha(x)\cdot \varphi(\gab(x))\cdot \hb(x)^{-1},
                                                   \tag{4.10$'$}
\end{equation}
for every $x\in \Uab$ and $\gra,\grb\in I$; and
\begin{equation}
   f(p)=\tea(f_o(\pi(p)))\cdot\ha(\pi(p))\cdot\varphi(\ga(p));
    \qquad p\in\pi^{-1}(U_{\gra}). \tag{4.11$'$}
\end{equation}
\end{Rem}

\medskip From the proof of Theorem~\ref{T42},  we single out the next useful result (with the notations of the same theorem).

\begin{Cor}   \label{C44}                %corollary 4.4
Every morphism $\varphi \col G\ra H$ of (Banach-) Lie groups is completed to a morphism $\ffi$ of $\pfb$ into $(Q,H,B,\pi')$ if and only if there is a family of smooth maps $\{\ha\col U_{\gra}\ra H\,|\,\gra\in I\}$ such that
\[
      \hab = \ha\cdot (\varphi \circ \gab)\cdot \hb^{-1},
\]
on $\Uab$, for every $\gra,\grb\in I$.
\end{Cor}

Theorem~\ref{T42} implies the following result, already proved in \cite{Vas1}:

\begin{Cor}   \label{C45}                %corollary 4.5
Two connections $\omega$ and $\theta$ on $(P,G,B,B)$ and $(Q,G,B,\pi')$, respectively, are $(f,\id_G,\id_B)$-related if and only if there is a family of smooth maps $\{\ha\col U_{\gra}\ra G\,|\,\gra\in I\}\}_{\gra\in I}$, such that
equalities
\begin{gather*}
   \hab = \ha\cdot \gab\cdot \hb^{-1},\\
  \omega_{\gra} = \Ad(h^{-1}_{\gra})\theta_{\gra}
            + h^{-1}_{\gra}dh_{\gra}
\end{gather*}
hold for all $\gra,\grb\in I$.
\end{Cor}

Applying once again Theorem~\ref{T41}, we recover \cite[Proposition 6.2]{Kob-Nom}, whose proof in the infinite-dimensional context is substantially more than a direct translation of its counterpart given in the aforementioned citation.

Before proceeding to details, we recall that the \emph{right Maurer-Cartan} or \emph{logarithmic differential} of a smooth map $f\col U\subseteq B\ra G$ is the form $df.f^{-1} \in \grL(U,\mathfrak{g})$, given by
\[
  \left(df.f^{-1}\right)_x(v):= \left(T_{f(x)}\rho_{f(x)^{-1}}
   \circ T_xf\right)(v); \qquad  x\in U,\, v\in T_xB,
\]
where $\rho_g\col G\ra G$ is the right translation of $G$ by $g\in G$ (the dot inserted in the notation of this differential is necessary in order to avoid any possible confusion here).

It can be verified that
\begin{gather*}
     df^{-1}.f = \Ad(f^{-1}).(df.f^{-1}),\\
      f^{-1}df = \Ad(f^{-1}).(df.f^{-1}).
\end{gather*}

\begin{Thm} \label{T46}                   %theorem 4.6
If $\ffi$ is a morphism of $P\equiv \pfb$ into $Q\equiv (Q,H,B,\pi')$, then every  connection $\omega$ on $P$ induces a unique connection $\theta$ on $Q$ which is $\ffi$-related with $\omega$.
\end{Thm}

\begin{proof} As in the previous cases, we may consider trivializations over the same open cover $\{U_{\gra}\}_{\gra\in I}$ of $B$. The connection (form) $\omega$ is completely known from its local connection forms $\{\oa\}_{\in I}$, according to \eqref{Eq35}. If $\{\ha\}_{\gra\in I}$ are the smooth maps of Theorem~\ref{T41}, we define the local forms
\begin{equation}  \label{Eq414}                  %(4.14)
  \ta := \Ad(\ha).\overline\varphi\oa -
     d\ha.\ha^{-1}\in \Lambda^1(U_{\gra},\mathfrak{h}),
                                           \qquad \gra\in I.
\end{equation}
We will prove the analog of \eqref{loccompcond}, namely
\begin{equation}      \label{thetacomp}                   %(4.15)
  \tb = \Ad({\hab}^{-1}).\ta +
                      {\hab}^{-1}d\hab; \qquad \gra,\grb\in I,
\end{equation}
ensuring the existence of a connection form $\theta$ on $Q$. To this end we elaborate the summands in the right-hand side of
\begin{equation}  \label{Eq416}                  %(4.16)
  \tb := \Ad(\hb).\overline\varphi\ob - d\hb.\hb^{-1}.                                      \end{equation}
First observe that, in virtue of \eqref{loccompcond}, \eqref{Eq45} and \eqref{Eq46}, equality
\begin{align*}
 \overline\varphi\ob&= \overline\varphi\left(\Ad(\gab^{-1}).\oa +
                                           \gab^{-1}d\gab\right)\\
   &= \Ad\big((\varphi\circ \gab)^{-1}\big). \overline\varphi\oa +
              (\varphi\circ \gab)^{-1}d(\varphi \circ \gab)
\end{align*}
holds on $\Uab$, thus
\begin{equation}   \label{Eq417}                    %(4.17)
   \begin{split}
 \Ad(\hb).\overline\varphi\ob &= \Ad\big(\hb\cdot(\varphi\circ
                              \gab)^{-1}\big).\overline\varphi\oa\\
   &\quad + \Ad(\hb).\big((\varphi\circ \gab)^{-1}d(\varphi \circ \gab)\big).   \end{split}
\end{equation}
But, for arbitrary smooth maps $f,g\col \Uab\ra H$, routine computations yield
\begin{equation}   \label{Eq418}                    %(4.18)
  \big(f\cdot g^{-1}\big)^{-1}d \big(f\cdot g^{-1}\big) =
                                    -dg.g^{-1} + \Ad(g).(f^{-1}df).
\end{equation}
Hence, applying \eqref{Eq418} for $g=\hb$, $f=\varphi\circ \gab$, and substituting the second summand in the right-hand side of \eqref{Eq417},  we transform the latter into
\begin{equation*}
  \begin{split}
  \Ad(\hb).\overline\varphi\ob &= \Ad\big(\hb\cdot(\varphi\circ
       \gab)^{-1}\big).\overline\varphi\oa + d\hb.\hb^{-1}\\
       &\quad + \big((\varphi\circ \gab)\cdot\hb^{-1}\big)^{-1}
                       d\big((\varphi\circ \gab)\cdot\hb^{-1}\big),
  \end{split}
\end{equation*}
by which \eqref{Eq416} becomes
\[
   \tb = \Ad\big(\hb\cdot(\varphi\circ
                              \gab)^{-1}\big).\overline\varphi\oa
       +  \big((\varphi\circ \gab)\cdot\hb^{-1}\big)^{-1}
                       d\big((\varphi\circ \gab)\cdot\hb^{-1}\big).
\]
Taking into account \eqref{Eq414} and the interplay between left and right differentials mentioned before the statement, the preceding equality takes the form
\begin{equation} \label{Eq419}                %(4.19)
  \begin{aligned}
  \tb &= \underbrace{\Ad\big(\hb\cdot(\varphi\circ
                       \gab)^{-1}\cdot \ha^{-1}\big)\ta}_{X}\\
      &+ \underbrace{\Ad\big(\hb\cdot(\varphi\circ
                           \gab)^{-1}\big).(\ha^{-1}d\ha)}_{Y}\\
      &+ \underbrace{\big((\varphi\circ \gab)\cdot\hb^{-1}\big)^{-1}
                    d\big((\varphi\circ \gab)\cdot\hb^{-1}\big)}_{Z}.
  \end{aligned}
\end{equation}
Using equality \eqref{Eq410}, we find that
\[
   X = \Ad(\hab^{-1}).\ta,\quad
   Y = \Ad(\hab^{-1}\cdot\ha).\ta(\ha^{-1}d\ha),
\]
and, by the analog of \eqref{Eq418},
\[
   Z = \big(\ha^{-1}\cdot\hab\big)^{-1}d(\ha^{-1}\cdot\hab)
     = \hab^{-1}d\hab - \Ad(\hab^{-1}\cdot\ha).(\ha^{-1}d\ha).
\]
Replacing the expressions of $X$, $Y$ and $Z$ in \eqref{Eq419}, we finally obtain the desired condition \eqref{thetacomp}. On the other hand, \eqref{Eq414} leads to
\[
  \overline\varphi\oa = \Ad(\ha^{-1}).\ta +
                           \ha^{-1} d\ha; \qquad \gra\in I,
\]
from which, in virtue of Theorem~\ref{T41}, follows that $\omega$ and $\theta$ are $(f,\varphi,\id_B)$-related.

It remains to show the uniqueness of $\theta$. The easiest way to check it is to use the corresponding splittings $C$ and $C'$ of the connections. More explicitly, since $\omega$ and $\theta$ are $(f,\varphi,\id_B)$-related,  by the equivalent conditions \eqref{Eq41} we have that
\[
C'\circ (f\times \id_B) =  Tf \circ C.
\]
Analogously, if there is another connection $\omega'' \equiv C''$, also $(f,\varphi,\id_B)$-related with $\omega$, then
\[
     C''\circ (f\times \id_B) =  Tf \circ C,
\]
thus $C''\circ (f\times \id_B)= C'\circ (f\times \id_B)$; that is,
\begin{equation}  \label{Eq415}                 %(4.15)
   C''(f(p),v)=C'(f(p),v);\qquad (p,v)\in P\times_B TB.
\end{equation}
Take now any $(q,u)\in Q\times_B TB$ and set $x:=\pi'(q) =\tau_B(u)$. Choosing an arbitrary $p\in \pi^{-1}(x)$, we can find an $h\in H$ such that $q=f(p)\cdot h$. As a result, in virtue of the action of $H$ on $(\pi')^*(TB)$, the $H$-equivariance of the splittings $C'$, $C''$, and \eqref{Eq415}, we see that
\begin{align*}
 C''(q,u)&= C''(f(p)\cdot h,u) = C''(f(p),u)\cdot h\\
   &:= T_{f(p)\cdot h}R'_{h}\big(C''(f(p),u)\big) =
       T_{f(p)\cdot h}R'_{h}\big(C'(f(p),u)\big)\\
   &=  C'(f(p)\cdot h,u) = C'(q,u),
\end{align*}
which concludes the proof.
\end{proof}

\section{Applications}  \label{S5}         %section 5

\subsection{\rm \underline{Connections on associated principal
bundles}}                        \label{Ss51} \  %\subsection{5.1}

\medskip\noindent We fix a principal bundle $P\equiv(P,G,B,\pi)$ endowed with a connection $\omega$ on it. A morphism $\varphi \col G\ra H$ of (Banach-) Lie groups determines the associated principal bundle $(P\times^G H, H,B,\pi_H)$  and the canonical morphism
\begin{gather*}
(\kappa,\varphi,\id_B)\col (P,G,B,\pi) \lra (P\times^G H, H,B,\pi_H)\col\\
                   p\longmapsto \kappa(p)=[(p,e)],
\end{gather*}
with $e\in H$ denoting the unit (neutral) element (see \cite[\S\S\,6.5--6.6, 7.10]{Bourbaki1}). Considering trivializations over a common open cover $\cU=\{U_{\gra}\}_{\gra\in I}$ of $B$, we check that the natural sections
$\{s_{\gra}^{\varphi}\}$ of the associated bundle are given by
\[
   s_{\gra}^{\varphi}(x) = (\kappa\circ s_{\gra})(x); \qquad x\in U_{\gra},
\]
where $\{s_{\gra}\}$ are the natural sections of  $P$. If $\{\oa\}$ are  the local connection forms of $\omega$, then

\medskip\hfill \begin{minipage}{15cm}
\emph{the 1-forms $\omega^{\varphi}_{\gra}:=\overline\varphi\oa$ ($\gra\in I$) are the local connection forms of a uniquely determined connection $\omega^{\varphi}$ on $P\times^G H$, which is $(\kappa,\phi,\id_B)$-related with $\omega$.}
\end{minipage}

\medskip\noindent This is indeed the case because the maps $\{\ha\}$ of Theorem~\ref{T41} are now constantly equal to the unit $e\in H$. The preceding result is verified also by Theorem~\ref{T46}.

\smallskip Specializing  to the particular case of a representation of $G$ into a Banach space $\E$, in other words, to a Lie group morphism $\varphi \col G\ra \mathrm{GL}(\E)$, we obtain the corresponding associated principal bundle
\[
  \varphi(P)=\left(P_{\varphi}:=P\times^G \mathrm{GL}(\E),
                              \mathrm{GL}(\E),  B,\pi_{\varphi}\right)
\]
with the connection $\omega^{\varphi}$, as in the previous general case. Note that in this case $\kappa(p)=[(p,\id_{\E})]$.

\smallskip \noindent Besides, a representation $\varphi \col G\ra \mathrm{GL}(\E)$ determines also the associated vector bundle $(E_{\varphi},B,\pi_E)$, where $E_{\varphi}=P\times^G\E$, inducing in turn its principal bundle of frames $(P(E_{\varphi}),\mathrm{GL}(\E),B,\bar\pi)$. The bundles $P$ and $P(E_{\varphi})$ are connected by the canonical morphism $(F,\varphi,\id_B)$ with
\[
   F\col P \lra P(E_{\varphi}) \col p \mapsto (x,\wt p),
\]
where $x:=\pi(p)$ and $\wt p \col \E\ra (E_{\varphi})_x=\bar\pi^{-1}(x)$ is the Banach space isomorphism given by $\wt p(u):=[(p,u)]$, $u\in \E$.
The local structure of $P(E_{\varphi})$ determines the corresponding natural local sections $\{\grs_{\gra}\}$, such that
\[
   \grs_{\gra}(x) = F(\sa(x));  \qquad x\in U_{\gra}, \gra\in I.
\]
Consequently, we check that

\medskip\hfill \begin{minipage}{15cm}
\emph{the 1-forms $\bar\omega_{\gra}:=\overline\varphi\oa$ ($\gra\in I$) are the local connection form of a uniquely determined connection $\bar\omega$ on $P(E_{\varphi})$, which is $(F,\phi,\id_B)$-related with $\omega$. Again the maps $\ha$ of Theorem~\ref{T41} are constantly equal to the unit $e\in H$.}
\end{minipage}

\medskip\noindent  This result is also in accordance with Theorem~\ref{T46}.

\smallskip By a sort of universal property of associated bundles, confirmed also by the local structures of $P_{\varphi}$ and $P(E_{\varphi})$ producing the same cocycles, the latter two bundles are  $\mathrm{GL}(\E)$-$B$-isomorphic. As a matter of fact, the map
\[
   \chi([(p,h)]) := F(p)\cdot h; \qquad [(p,h)\in P\times^G H,
\]
induces the principal bundle isomorphism $(\chi,\id_{\mathrm{GL}(\E)},\id_B)$ making the following diagram commutative:
\begin{diagram}
    P &   \rTo^{\kappa\quad}  &
                        P_{\varphi}=P\times^G \mathrm{GL}(\E)\\
      & \rdTo_{\rotatebox{40}{$F$}} &
                      \cong\dDashto_{\raisebox{.5ex}{$\chi$}}\\
      &                              & P(E_{\varphi})
\end{diagram}

\smallskip\centerline{{\sc Diagram 3}}

\medskip\noindent Since $\oa^{\varphi}=\bar\omega_{\gra}$ and $\ha=e$, for every $\gra\in I$, we conclude that

\medskip\hfill \begin{minipage}{15cm}
\emph{the connections $\omega^{\varphi}$ and $\bar\omega$ are $(\chi,\id_{\mathrm{GL}(\E)},\id_B)$-related.}
\end{minipage}

\subsection{\rm \underline{Connections on the frame bundle}}
                              \label{Ss52}  \   %\subsection 5.2

\medskip\noindent Let $E \equiv (E,B,\pi_E)$ be a vector bundle of fibre type $\E$ and the corresponding principal bundle of frames $P(E)\equiv (P(E),\mathrm{GL}(\E),B, \pi_P)$. By appropriate restrictions, we may consider a smooth atlas $\{(U_{\gra},\phi_{\gra})\}_{\gra\in I}$ of $B$, over which we define the local trivializations of both $E$ and $P(E)$.

Assume first that $E$ admits a linear connection $K$ with Christoffel symbols the smooth maps (see \cite{Flaschel})
\[
 \Grc_{\gra}\col \phi_{\gra}(U_{\gra})\lra \cL_2(\B,\E;\E),
                                                   \qquad \gra\in I.
\]
For every $x\in U_{\gra}$ and every $v\in T_xB$, we set
\begin{equation}   \label{5.1}                      %(5.1)
  \omega_{\gra,x}(v) := \Grc_{\gra}(\phi_{\gra}(x))\big(
                     T_x\phi_{\gra}(v), \boldsymbol{\cdot}\big).
\end{equation}
More precisely, for $x, v$ as before, and every $u\in \E$,
\[
  \omega_{\gra,x}(v).u = \Grc_{\gra}(\phi_{\gra}(x))
                                      \big(T_x\phi_{\gra}(v), u\big).
\]
The forms $\{\oa\in \grL^1(U_{\gra},\cL(\E))\,|\, \gra\in I\}$ determine a connection form $\omega$ on $P$, since the compatibility of the Christoffel symbols
\begin{equation}  \label{Chriscomp}                %(5.2)
 \begin{aligned}
\Grc_{\gra} (x) &=  G_{\gra\grb}\left((\fb\circ\fa^{-1})(x)\right)
                                         \circ \big[DG_{\grb\gra} + \\
    &\qquad  \Grc_{\grb}\left((\fb\circ\fa^{-1})(x)\right)
                \circ \left(D(\fb\circ\fa^{-1})(x)\times
                                    \Grc_{\grb\gra}(x)\right)\big]
 \end{aligned}
\end{equation}
implies the compatibility condition \eqref{loccompcond}. The smooth maps
$G_{\gra\grb} \col \phi_{\grb}(\Uab)\ra \mathrm{GL}(\E)$ figuring in \eqref{Chriscomp} are given by $G_{\gra\grb}(\phi_{\grb}(x)):=\Phi_{\gra,x}\circ \Phi_{\grb,x}^{-1}$, $x\in \Uab$, where $\Phi_{\gra,x}\col E_x\ra \E$ ($\gra\in I$) is the toplinear isomorphism induced by the trivialization $(U_{\gra},\Phi_{\gra})$ of $E$.

Conversely, assume that $P(E)$ admits a connection with connection form $\omega$ and local connection forms $\{\oa\}_{\gra\in I}$. Setting
\begin{equation} \label{5.3}                       %(5.3)
   \Grc_{\gra}(z)(y, u) := [(\psi_{\gra}^*\omega_{\gra})_z](u)
 \end{equation}
for every $z\in \phi_{\gra}(U_{\gra})$, $y\in \B$ and $u\in \E$, with $\psi_{\gra} = \phi_{\gra}^{-1}$, we check that \eqref{loccompcond} implies \eqref{Chriscomp}, thus $\{\Grc_{\gra}\}$ determine a linear connection with Christoffel symbols the previous family.

As a result, the linear connections on $E$  are in bijective correspondence with the connections on $P(E)$. The same local tools can be used to connect related connections on vector bundles and the corresponding connections on the bundles of frames. However, this point of view is omitted because it relies on material outside the scope of the present note. For detailed discussions on this matter we refer to \cite{Vas2}, \cite{Vas3} and \cite{Vas4}.

\subsection{\rm\underline{Connections on inverse limit bundles}}
                           \label{Ss53} \     %\subsection{5.3}

\medskip\noindent We start with a few preliminary notions, referring for further details to \cite{Gal1}, \cite{Gal2}, \cite{Dod-Gal-Vas}:

Let $\{M^{i};\mu ^{ji}\}_{i,j\in \mathbb{N}}$ be a projective system of smooth manifolds modelled on the Banach spaces $\{\mathbb{E}^{i}\}_{i\in \mathbb{N}}$, respectively. A system of corresponding charts $\{(U^{i},\phi ^{i})\}_{i\in \mathbb{N}}$ will be called a \emph{(projective) limit chart} if and only if the limits $\varprojlim U^{i}$, $\varprojlim \phi ^{i}$ exist, and $\varprojlim U^{i}$, $\varprojlim \phi^{i}(\varprojlim U^{i})$ are open subsets of $\varprojlim M^{i}$, $\varprojlim \mathbb{E}^{i}$, respectively.

We call $M=\varprojlim \{M^{i};\mu ^{ji}\}_{i,j\in \mathbb{N}}$ a \emph{projective limit of Banach manifolds} (\emph{plb-manifold} for short), if the following conditions are satisfied:
\renewcommand{\labelenumi}{(\alph{enumi})}
\begin{enumerate}
        \item The model spaces $\{\mathbb{E}^{i}\}_{i\in \mathbb{N}}$ form a projective system with connecting morphisms
            \[ \rho^{ji}\col\mathbb{E}^{j}\lra \mathbb{E}^{i},
            \quad \forall\,i,j\in \N \; \text{with}\; j\geq i
            \]
            (thus the limit  $\mathbb{F}:=\varprojlim \mathbb{E}^{i}$ a Fr\'echet space).
        \item $M$ is covered by a family $\{(U_{\alpha},\phi _{a})\}_{\alpha \in I}$ of limit charts, where $U_{\alpha} = \varprojlim U_{\gra}^{i}$ and $\phi _{a}= \varprojlim \phi_{\gra}^{i}$ (the limits being taken with respect to $i\in \N$)
\end{enumerate}
Then $M=\varprojlim M^{i}$ has the structure of a Fr\'echet smooth manifold, modelled on $\mathbb{F}$. The necessary differential calculus in the models is based on the differentiation method of \cite{LE1} and \cite{LE2}.

\medskip Let  $\{G^{i};g^{ji}\}_{i,j\in \mathbb{N}}$ be a \emph{projective system of Banach-Lie groups} (\emph{plb-group}), i.e. a projective system of groups satisfying the conditions (a) and (b), where the connecting morphisms $g^{ji}\col G^j\ra G^i$ are Lie group homomorphisms. If $G^i$ is modelled on the Banach space $\mathbb{G}^i$ ($i\in \N)$, then the plb-manifold $G=\varprojlim G^{i}$ is a Lie group modelled on the Fr\'echet space $\mathbb{G}=\varprojlim \mathbb{G}^i$. Moreover, the Lie algebra $\mathfrak{g}$ of $G$, identified with
$\varprojlim T_{e^i}G^i = \varprojlim \mathfrak{g}^i$, is  a Fr\'echet space.

Assume that $(P,G,B,\pi)$ is principal bundle over a Banach space $B$, with structure group a plb-group as before. Then $P$ is a Fr\'echet  principal bundle such that $P=\varprojlim P^i$, and $\pi=\varprojlim \pi^i$, where $(P^i,G^i,B,\pi^i)$ is a Banach principal bundle, for every $i\in \N$. The local trivializations of $P$ are projective limits of appropriate trivializations of the factor bundles $P^i$.

A natural question here is how to define a connection (form) $\omega\in\Lambda^1(P,\mathfrak{g})$ on such a limit bundle, if each factor bundle $P^i$ is endowed with a connection $\omega^i\in \Lambda^1(P^i,\mathfrak{g}^i)$. It would be tempting to think that $\omega = \varprojlim \omega^i$. However, each $\omega^i$ is a (smooth) section of the linear map bundle $L(TP^i,\mathfrak{g}^i)$, $i\in \N$, thus the existence of a projective limit of the previous 1-forms cannot be expected, because the manifolds $\{L(TP^i,\mathfrak{g}^i)\}_{i,j\in \N}$ do not form, in general, a projective system.

We circumvent the impasse by assuming that, for every $j\geq i$, the connections $\omega^j$ and $\omega^i$  are \emph{$(p^{ji},g^{ji},\id_{B})$-related}, where $p^{ji}\col P^j\ra P^i$ are the connecting morphisms of the projective system $\{P^i\}_{i\in \N}$ generating $P$. Under this condition,

\smallskip\hfill \begin{minipage}{15cm}
\emph{the (point-wise) limit  $\omega(u) := \varprojlim (\omega^i(u^{i}))$ does exist, for every $u=(u^{i})_{i\in \N}\in P$, and $\omega$ is indeed a connection on $P$.}
\end{minipage}

\medskip\noindent The necessary technicalities and the detailed proofs of this claims are given in the aforementioned references.

Regarding the local connection forms, which are the object of our interest here, they enter the stage in the proof of the following statement, characterizing \emph{all} the connections of a projective limit bundle:

\medskip\hfill \begin{minipage}{15cm}
\emph{Every connection $\omega$ on a projective limit principal bundle $(P,G,B,\pi)$, as above, is necessarily the limit of a projective system of connections in the sense that $\omega(u) = \varprojlim (\omega^i(u^{i}))$, for every $u=(u^{i})_{i\in \N}\in P$.}
\end{minipage}

\medskip\noindent We outline the proof, actually based on a series of auxiliary results: Denote by $\left\{\sa\col U_{\gra}\ra P\right\} _{\gra\in I}$ the family of natural local sections of $P$ over a trivializing cover $\left\{ (U_{\gra},\Phi_{\gra})\right\}_{\gra\in I}$, and by
\[
   p^i \col P=\varprojlim P^i\lra P^i, \quad
   g^i \col G=\varprojlim G^i\lra G^i
\]
the canonical projections of the corresponding projective limits. It is proved that
\[
    L(TU_{\gra},\mathfrak{g})
                              =\vpl L(TU_{\gra},\mathfrak{g}^i),
\]
with projections
\[
    \grl^{i}\col L(TU_{\gra},\mathfrak{g})\ra
        L(TU_{\gra},\mathfrak{g}^i)\col f\mapsto T_{e}g^i\circ f,
                                                      \qquad i\in\N.
\]
Note that the preceding limit exists, because the linear map bundles involved have the same domain  $TU_{\gra}$, in contrast to the problematic $\{L(TP^i,\mathfrak{g}^i)\}_{i,j\in \N}$ discussed earlier. As a consequence, the local connection forms $\{\oa =\sa^{\ast }\omega\}_{\gra\in I}$ of $\omega$ induce the (smooth) $\mathfrak{g}^i$-valued 1-forms
\[
   \oa^{i}:=\grl^i\circ \oa\col U_{\gra}\ra
                  L(TU_{\gra},\mathfrak{g}^i);\qquad i\in \N,
\]
satisfying the ordinary compatibility condition of local connection forms in the corresponding bundle $P^i$. Therefore, each $P^i$ admits a connection $\omega^i$, with local connection forms precisely the given family $\{\oa^{i}\}_{\gra\in I}$. Moreover, $\omega^j$ and $\omega^i$ are $(p^{ji},g^{ji},\id_B)$-related, for every $i,j\in \N$ with $j\geq i$. The latter condition ensures that the 1-form $\bar\omega\in \Lambda^i(TP,\mathfrak{g})$, defined by
$\bar\omega(u) = \varprojlim (\omega^i(u^{i}))$, for every $u=(u^{i})_{i\in \N}\in P$, is a connection on $P$. We conclude that $\omega=\bar\omega$, since both connections have the same local connection forms.

\subsection{\rm\underline{Sheafification of local connection forms}}
                           \label{Ss54} \     %\subsection{5.4}

\medskip\noindent As local objects, the local connection forms are susceptible of a sheaf-theoretic globalization. For the necessary terminology and elementary properties of sheaf theory we refer, e.g., to \cite{Bredon}, \cite{Godement}, \cite{Warner}. More technical details on the subject matter of this subsection can be found  in \cite{Vas5}.

Starting with a finite-dimensional  principal bundle $(P,G,B,\pi_P)$, we consider the \emph{sheaf of germs of smooth sections of $P$}, denoted by $\cP \equiv (\cP,B,\pi)$. This is the sheaf generated by the presheaf $U\mapsto \Gamma(U,P)$, where $\Gamma(U,P)$ is the set of smooth sections of $P$ over $U$, with $U$  running the topology $\mathcal{T}_B$ of $B$. We also write
\[
      \cP = \mbox{\bf S}\big(U\longmapsto \Gamma(U,P)\big),
\]
where \mbox{\bf S} is the \emph{sheafification functor}. Since $\cP$ is a functional sheaf, thus \emph{complete}, it follows that
\begin{equation}   \label{Eq5.4}                               %(5.4)
     \cP(U) \cong \Gamma(U,P).
\end{equation}
The left-hand side of \eqref{Eq5.4} stands for the set of sections of $\cP$ over $U$. As  a matter of fact, to an $s\in \Gamma(U,P)$ there corresponds bijectively the section $\grs\equiv\tilde s\in\cP(U)$ with  $\grs(x)=\tilde s(x):=[s]_x\in \cP_x$, for every $x\in U$.

For convenience, we denote by $\cA$ the \emph{sheaf of germs of smooth $\R$-valued functions on $B$}, usually written as $\cC_B^{\infty}$. As in \eqref{Eq5.4},
\begin{equation}   \label{Eq55}                               %(5.5)
     \cA(U) = \cC_B^{\infty}(U) \cong C^{\infty}(U,\R),
                                          \qquad U\in \mathcal{T}_B.
\end{equation}

Likewise, we denote by $\cG$ (resp $\cL$) the \emph{sheaf of germs of smooth $G$-valued \emph(resp. $\G$-valued) maps on B}, where $\G\equiv T_eG$ is the Lie algebra of $G$, thus
\begin{equation}   \label{Eq56}                               %(5.6)
      \cG(U)\cong C^{\infty}(U,G), \quad \cL(U)\cong C^{\infty}(U,\G),
                                             \qquad U\in \mathcal{T}_B.
\end{equation}

There is a natural action of the sheaf of groups $\cG$ on the right of $\cP$, induced by the sheafification of the (presheaf) actions
\[
   \Gamma(U,P) \times C^{\infty}(U,G) \lra
   \Gamma(U,P) \col (s,g) \mapsto s\cdot g,
\]
for every $U\in \mathcal{T}_B$. Moreover, the local structure of the initial bundle $P$ implies that $\cP|_{U_{\gra}} \cong \cG|_{U_{\gra}}$, for an open cover $\cU = \{\Ua \sst B \, | \, \gra \in I\}$ of $B$ over which $P$ is trivial. Thus, according to the terminology of \cite{Grothendieck}, $\cP$ is a \emph{principal sheaf of structure type $\cG$ with structural sheaf $\cG$}.

Similarly, we obtain the sheaves of germs of ordinary ($\R$-valued)  and $\G$-valued 1-forms on $B$, denoted by $\Om$ and $\OmG$, respectively. Then
\begin{equation}   \label{Eq57}                               %(5.7)
 \OmG(U) \cong \grL^1(U,\G) \cong \Om(U)\otimes_{\cA(U)}\cL(U)
\end{equation}
(for the last isomorphism see, for instance, \cite[p.~81]{Halperin}). But $\cL\cong \cA^n$ ($n=\dim \G$), therefore
\[
 \Om(U)\otimes_{\cA(U)}\cL(U)\cong \left(\Om\otimes_{\cA}\cL\right)(U),
\]
from which, together with \eqref{Eq57}, we obtain
\begin{equation}   \label{Eq58}                               %(5.8)
 \OmL := \Om\otimes_{\cA}\cL \cong\OmG.
\end{equation}

With the notations of \eqref{Eq33}, we define the map
\begin{equation}   \label{Eq59}                               %(5.9)
  \partial_U \col C^{\infty}(U,G)\lra \grL^i(U,\G) \col
       f\mapsto\partial_U(f):=f^{-1}df,
\end{equation}
for every $U\in\mathcal{T}_B$. The properties of the Maurer-Cartan differential imply that
\begin{equation}   \label{Eq510}                               %(5.10)
 \partial_U(g\cdot h)=\Ad(h^{-1}).\partial_U(g)+\partial_U(h);
                                     \qquad g,h \in C^{\infty}(U,G).
\end{equation}
As a result, in virtue of the identifications \eqref{Eq56}--\eqref{Eq58}, we obtain the morphism of presheaves
\[
  \left\{\partial_U\col \cG(U)\lra\OmL(U)\right\}_{U\in\mathcal{T}_B},
\]
generating the morphism of sheaves of sets
\[
       \partial \col  \cG\lra \OmL.
\]
In addition, the ordinary adjoint representation determines the morphism of sheaves
\[
  \cA d\col \cG\lra \cA ut(\cL),
\]
generated by the morphisms
\[
   \Ad_U\col C^{\infty}(U,B) \lra \Aut(\cL|_U),
\]
defined in turn by
\[
   \big(\Ad_U(g)(f)\big)(x):= \big(\Ad(g(x))\big)(f(x)),
\]

\noindent for every $g\in C^{\infty}(U,B)$, $f\in\cC^{\infty}(V,\G)$, $x\in V$, and every open $V\sst U$.

By means of $\cA d$, the sheaf of groups $\cG$ acts in a natural way on (the left) of $\OmL$: For every open $U\sst B$, $g\in \cG(U)$, and any decomposable element $w\otimes u\in \Om(U)\otimes_{\cA(U)}\cL(U)\cong \OmL(U)$, we set
\[
   \cA d(g)\cdot (w\otimes u):= w \otimes \cA d(g)(u).
\]
Here we have employed the usual convention to denote by the same symbol both a sheaf morphism and the induced morphism between sections.  The previous equality, by linear extension, determines $\cA d(g)\cdot\theta$, for every $g\in \cG(U)$ and $\theta\in \OmL(U)$, thus leading to the desired action. Consequently, taking into account the definition of $\partial$, we obtain the equalities:
\begin{gather}                     %(5.11)-(5.12)
  \partial(a\cdot b) = \cA d(b^{-1})\cdot \partial(a) + \partial(b);
        \qquad (a,b)\in \cG\times_B\cG,\label{Eq511}\\
  \partial(a^{-1}) = \cA d(a)\cdot \partial(a), \qquad a\in\cG.
                                           \label{Eq512}
\end{gather}

To state the main results of this subsection, we denote by $\grs_{\gra} \in \cP(U_{\gra})$ the sections of $\cP$ corresponding to the natural sections $\sa\in \Gamma(U_{\gra},P)$ of $P$ by \eqref{Eq5.4}, and by $\gamma_{\gra\grb}\in \cG^{\infty}(U_{\gra\grb})$ the sections of $\cG$ corresponding to the transition functions $\gab\in \cC^{\infty}(U_{\gra\grb},G)$ of $P$. Finally, we denote by $\theta_{\gra}\in \OmL(U_{\gra})$ the sections corresponding to $\oa\in\grL^1(U,\G)$ by the identifications \eqref{Eq57}  and \eqref{Eq58}. It is worth noticing that the \emph{transition sections} $\{\gamma_{\gra\grb}\}_{\gra,\grb\in I}$ fully determine the structure of $\cP$.

With the previous notations in mind, the compatibility of the ordinary local connection forms $\{\omega_{\gra}\}$ transcribes to
\begin{equation}  \label{shloccompcond}                 %(5.13)
 \theta_{\grb} = \cA d(\gamma_{\gra\grb}^{-1}).\theta_{\gra} +
                  \partial(\gamma_{\gra\grb}), \qquad \gra,\grb \in I.
\end{equation}

\begin{Thm}
A family of sections $\{\theta_a\in \OmL(U_{\gra})\,|\,\gra\in I\}$, satisfying \eqref{shloccompcond}, determines a unique sheaf morphism $D\col \cP\ra \OmL$ such that
\begin{equation}  \label{Dequivar}                      %(5.14)
   D(w\cdot a)=  \cA d(a^{-1}).D(w) + \partial(a);
                                     \qquad (w,a)\in \cP\times_B\cG,
\end{equation}
and $D(\grs_{\gra})= \theta_{\gra}$, for every $\gra\in I$.                                     \end{Thm}

Before sketching the proof, let us observe that \eqref{Dequivar} describes the behavior of $D$ with regard to the action of $\cG$ on $\cP$, thus it can be thought of as the sheaf analog of \eqref{w2} in Section 2.

On the other hand, since it is easily proved that $D$ corresponds bijectively to a connection $\omega\equiv \{\oa\}$ on $P$, $D$ is called (by abuse of language)  a \emph{connection on $\cP$} (even on $P$!), and the sections $\{\theta_{\gra}\}$ are said to be the \emph{local connection forms of $D$}. The morphism $D$ gives an operator-like definition of principal connections. For a sheaf-theoretic description of connections on vector bundles we refer to \cite{Wells}.

\begin{proof}
Given the local sections $\{\theta_{\gra}\}$, we define  $D$ as follows: If $w$ is an arbitrarily chosen element of $\cP$ with $p(w)=x\in U_{\gra}$,
we set
\[
   D(w):= \cA d(b^{-1}_{\gra}).\theta_{\gra}(x) + \partial(b_{\gra}),
\]
where $b_{\gra}$ is the unique element in the fiber $\cG_x$ satisfying equality $w=\grs_{\gra}(x)\cdot b_{\gra}$. The compatibility condition \eqref{shloccompcond} ensures that $D$ is well defined.

$D$ is a continuous map: For $w$ and $x$ as before,  there are open neighborhoods $V\subseteq \cP$ and $U\subseteq B$ of $w$ and $x$, respectively, such that $p|_V \col V\ra U$ is a homeomorphism. Denote by $\grs$ its inverse and set $W=\grs(U\cap U_{\gra})$. Then there exists a unique $g\in \cG(U\cap U_{\gra})$, such that
\[
   D|_W = \big(\cA d(g^{-1}).\theta_{\gra} +
                                      \partial g\big)\big|_W.
\]
This proves the continuity of $D$ at an arbitrary $w\in \cP$, hence its continuity on $\cP$.

Equality $D(\grs_{\gra})=\theta_{\gra}$ is obvious from the definition of $D$. The same equality ensures the uniqueness of $D$
\end{proof}

The converse statement is also true; that is, a sheaf  morphism satisfying \eqref{Dequivar} determines a family $\{\theta_{\gra}\}$ satisfying the compatibility condition \eqref{shloccompcond}. Therefore, sheaf morphisms $D$ correspond bijectively to local sections $\{\theta_{\gra}\}$ as before. Because both quantities correspond to ordinary connections  and local connection forms, respectively, the terminology induced before the proof is fully justified.

\smallskip To describe related connections, we proceed as follows: A morphism of principal bundles (with the same base)
\[
   (f,\phi,\id_B)\col (P,G,B,\pi) \lra (P',G',B,\pi')
\]
determines, by sheafification, the sheaf morphisms (same notations as before, for the sake of simplicity)
\begin{gather*}
   f\col \cP \lra \cP', \quad \phi \col \cG \lra \cG', \quad
                                      \ophi\col \cL\lra \cL'.
\end{gather*}
Hence, two connections $D$ and $D'$ on $\cP$ and $\cP'$, respectively, are said to be \emph{$(f,\phi,\id_B)$-related} if $D'\circ f = (1\otimes\ophi)\circ D$; in other words, the following diagram is commutative.
\smallskip
\begin{diagram}
     \cP       &    \rTo^{f}              & \cP' \\
     \dTo^{D}  &                          & \dTo_{D'} \\
 \OmL =\Om\otimes_{\cA} \cL &\rTo^{\quad 1\otimes\ophi\quad} &
                    \Om\otimes _{\cA} \cL' =\Omega^1(\cL')
\end{diagram}

\smallskip\centerline{{\sc Diagram 4}}

\begin{Thm} \label{T52}
Two connections $D$ and $D'$ are $(f,\phi,\id_B)$-related if and only if the corresponding local connection forms (sections) $\{\theta_{\gra}\}$ and $\{\theta'_{\gra}\}$ satisfy condition
\begin{equation}
   (1\otimes\ophi)(\theta_{\gra}) = \cA d(\ha^{-1}).\theta'_{\gra}
                                     +\partial(\ha); \qquad \gra\in I,
\end{equation}
where $\ha\in \cG'(U_{\gra})\cong C^{\infty}(U_{\gra},G')$ are determined by $f(\grs_{\gra})=\grs'_{\gra}\cdot \ha$.
\end{Thm}
The proof, very roughly, can be thought of as the sheafification of that of Theorem~\ref{T41}, taking into account certain subtleties regarding tensor products.

Similarly, we can state and prove the sheaf-theoretic analogs of Theorem~\ref{T42}, Corollary~\ref{C44} and Theorem~\ref{T46}. In the same vein, we obtain the analogs of the results of \S\,\ref{Ss51}, but in this case we need to introduce the notions of associated principal and vector sheaves (cf. \cite[Chapters 5,7]{Vas5}).

The results of this subsection hold true also in the infinite-dimensional (Banach) framework. However, computations involving the sheaf of germs of $\mathfrak{g}$-valued 1-forms on $P$, $\OmL$, should be carried out without use
of the tensor representation \eqref{Eq58}. The latter is valid only for \emph{locally free sheaves}, as is the case of the sheaves considered here, which are generated by finite-dimensional manifolds and Lie groups.

Concluding, we note that the present sheaf-theoretic approach to ordinary connections serves as a motivation (and a fundamental example) of the abstract theory of connections on arbitrary principal sheaves, elaborated in \cite{Vas5}.

\end{document}